\documentclass[11pt,twoside]{amsart}
\usepackage{amsmath}
\usepackage{amssymb}
\usepackage{amsthm}
\usepackage{latexsym}
\usepackage{amscd}
\usepackage{xypic}
\xyoption{curve}
\usepackage{ifthen}
\usepackage{hyperref}
\usepackage{graphicx}
\usepackage{enumerate}

 \def\cocoa{{\hbox{\rm C\kern-.13em o\kern-.07em C\kern-.13em o\kern-.15em A}}}

\usepackage{xcolor}

\newtheorem{theorem}{Theorem}[section]
\newtheorem{lemma}[theorem]{Lemma}
\newtheorem{proposition}[theorem]{Proposition}
\newtheorem{corollary}[theorem]{Corollary}

\theoremstyle{definition}
\newtheorem{remark}[theorem]{Remark}
\newtheorem{definition}[theorem]{Definition}

\newtheorem{notation}[theorem]{Notation}

\newtheorem{problem}[theorem]{Problem}

\newcommand {\Hom}{\mathcal{H}\kern -0.25ex{\mathit om}}
\newcommand {\Ext}{\mathrm{Ext}}

\newcommand {\ext}{\mathrm{ext}}

\newcommand {\Sing}{\mathrm{Sing}}

\newcommand {\Hilb}{\mathcal{H}\kern -0.25ex{\mathit ilb\/}}

\usepackage[none]{hyphenat}

\newcommand {\ft}{\mathfrak{t}}
\newcommand {\fm}{\mathfrak{m}}
\newcommand {\fe}{\mathfrak{e}}
\newcommand {\fb}{\mathfrak{b}}

\newcommand{\Pic}{\operatorname{Pic}}
\newcommand{\Num}{\operatorname{Num}}
\newcommand{\Supp}{\operatorname{Supp}}
\newcommand{\h}{\operatorname{h}}
\newcommand{\Ho}{\operatorname{H}}

\newcommand{\rank}{\operatorname{rank}}

\begin{document}
\title[ Brill-Noether  on ruled surfaces]
{Brill-Noether Theory of stable vector bundles on ruled surfaces}

\author[L.\ Costa, I. Macías Tarrío]{L.\ Costa$^*$, I. Macías Tarrío$^{**}$}

\address{Facultat de Matem\`atiques i Inform\`atica,
Departament de Matem\`atiques i Inform\`atica, Gran Via de les Corts Catalanes
585, 08007 Barcelona, SPAIN } \email{costa@ub.edu}

\address{Facultat de Matem\`atiques i Inform\`atica,
Departament de Matem\`atiques i Inform\`atica, Gran Via de les Corts Catalanes
585, 08007 Barcelona, SPAIN} \email{irene.macias@ub.edu}

\date{\today}
\thanks{$^*$ Partially supported by PID2020-113674GB-I00-*2151.}
\thanks{$^{**}$ Partially supported by PID2020-113674GB-I00-*2151.}

\subjclass{14J60, 14J26}

\begin{abstract}
Let $X$ be a ruled surface over a nonsingular curve $C$ of genus $g\geq0$.  Let $M_H:=M_{X,H}(2;c_1,c_2)$ be the moduli space of $H$-stable rank 2 vector bundles $E$ on $X$ with  fixed Chern classes $c_i:=c_i(E)$ for $i=1,2$. The main goal of this paper is to contribute to a better understanding of  the geometry of the moduli space $M_H$ in terms of its Brill-Noether locus $W_H^k(2;c_1,c_2)$, whose points correspond to stable vector bundles in $M_H$ having at least $k$ independent sections. We deal with the non-emptiness of this Brill-Noether locus, getting in most of the cases sharp bounds for the values of $k$ such that   $W_H^k(2;c_1,c_2)$  is non-empty. 
\end{abstract}


\maketitle

\tableofcontents

\section{Introduction}

Let $X$ be a smooth projective variety of dimension $n$ over an algebraically closed field $K$ of characteristic $0$ and let $M_H=M_{X,H}(r;c_1,\dots,c_s)$ be the moduli space of stable rank $r$ vector bundles $E$ with respect to an ample divisor $H$ on $X$ and with fixed Chern classes $c_i:=c_i(E)$ for $i=1,\dots, s:=\min\{r,n\}$.  

Moduli spaces of stable vector bundles were defined in 1970 by Murayama.
Since then, many authors have done a systematic study of these moduli spaces but in general, in spite of great contributions, a lot of questions remain open.

One way of studying the geometry of these moduli spaces is by studying their subvarieties.
A Brill-Noether locus $W_H^k=W_H^k(r;c_1,\dots, c_s)$ is a subvariety of $M_{X,H}(r;c_1,\dots,c_s)$ whose points correspond to stable vector bundles having at least $k$ independent sections and, roughtly speaking, the Brill-Noether Theory describes the geometry of these subvarieties.

Classical Brill-Noether Theory deals with line bundles (respectively rank $r$ stable bundles) on smooth curves. In the case of line bundles on projective curves $C$ of genus $g$, Brill-Noether theory is concerned with the subvarieties of $\Pic^{d}(C)$ whose line bundles have at least $k$ independent global sections. Basic questions concerning non-emptiness, connectedness, irreducibility, dimension, singularities, cohomological classes ... have been answered when the curve $C$ is generic on the moduli space of curves of genus $g$. 

There are several natural ways to generalize the classical theory of Brill-Noether.
The first natural generalization concerns higher rank vector bundles on algebraic curves. 
During the last two decades, a great amount of job has been made around the Brill-Noether stratification of the moduli space of  stable rank $r$ vector bundles of degree $d$ on algebraic curves, giving rise to nice and interesting descriptions of these subvarieties. Nevertheless, it should be mentioned that in spite of big efforts, many questions concerning their geometry still remain open. The other way of generalizing this theory is considering line bundles on varieties of arbitrary dimension. Finally, we can go in both directions simultaneously.
 The Brill-Noether theory for rank $r$ bundles on a variety of arbitrary dimension was stablished in \cite{Costa}. If fact, if $X$ is a  smooth projective variety  of dimension $n$ and $H$ an ample divisor on $X$, the Brill-Noether locus $W_H^k(r;c_1,\dots, c_s)$ in $M_{X,H}(r;c_1,\dots, c_s)$ is defined as the set of stable vector bundles in $M_{X,H}(r;c_1,\dots, c_s)$ having at least $k$ independent sections. It is proved in \cite{Costa} that $W_H^k(r;c_1,\dots, c_s)$ has a natural structure of $k$-determinantal variety and that these subvarieties give a filtration of the moduli space.

\vspace{0.2cm}
In the context of the Brill-Noether Theory, one can immediately raise the following natural problems: 
\begin{problem}
\label{problembn}
    \begin{itemize}
        \item[(a)] Determine when the Brill-Noether locus $W_H^k$ is non-empty.
        \item[(b)] Study the smoothness of the Brill-Noether locus.
        \item[(c)] See if it is true that  $\rho_H^k>0$ implies that $W_H^k\neq \emptyset$ and $\rho_H^k<0$ implies that $W_H^k=\emptyset$.
        \item[(d)] Determine the irreducibility of $W_H^k$.
    \end{itemize}
\end{problem}

In this paper, we will focus the attention on moduli spaces of rank $2$ stable vector bundles on ruled surfaces.
In \cite{Qin1} Qin developed a theory of walls and chambers and proved that moduli spaces $M_{X,H}(2;c_1,c_2)$ only depend on the chamber of $H$. In this paper we will analize how this phenomena affects the Brill-Noether loci and in addition using this structure we will obtain our first contribution to Problem \ref{problembn} (a) (see Theorem \ref{teowall}).

 Concerning Problem \ref{problembn} (a) we will also prove in Theorem \ref{teo}:

\begin{theorem}

Let $X$ be a ruled surface over a nonsingular curve $C$ of genus $g\geq0$, $m\in\{0,1\}$, $c_2>>0$ an integer and $H\equiv\alpha C_0+\beta f$ an ample divisor on $X$ with $$\alpha (e+m)<\beta.$$ Then, for any $k$ in the range $$\max\{1,g\}\leq  k<\frac{1}{2\alpha}[\beta -\alpha (e- m+2g-2 )],$$ the Brill-Noether locus $W_H^k(2;C_0+\fm f,c_2)\neq\emptyset$ with $m=\deg(\fm)$.
\end{theorem}
\vspace{0.1cm}

The fact that $W_H^k(2;C_0+\fm f, c_2)\neq\emptyset$ for all $\max\{1,g\}\leq k< \frac{1}{2\alpha}[\beta-\alpha(e-m+2g-2)]$ implies that  $W_H^k(2;C_0+\fm f, c_2)\neq\emptyset$ for all $1\leq k< \frac{1}{2\alpha}[\beta-\alpha(e-m+2g-2)]$.
Moreover, we will see that this bound is sharp (see Theorem \ref{thvacio}).

\vspace{0.2cm}
On the other hand we will prove:

\begin{theorem}

Let $X$ be a ruled surface over a nonsingular curve $C$ of genus $g\geq0$,  $c_2>>0$ an integer  and $H\equiv C_0+\beta f$ an ample divisor on $X$.
Then  $W_H^1(2;f,c_2)\neq\emptyset$ and $W_H^k(2;f,c_2)=\emptyset$ for all $k\geq3$.
\end{theorem}

Finally, we will study in detail the locus $W_H^1(2;f,c_2)\backslash W_H^2(2;f,c_2)$ and we will contribute to Problem \ref{problembn} (b) showing that:
\begin{proposition}

Let $X$ be a ruled surface over a nonsingular curve $C$ of genus $g\geq0$, $c_2>>0$ an integer and $H\equiv C_0+\beta f$  an ample divisor on $X$. Then $W_H^1(2;f,c_2)\backslash W_H^2(2;f,c_2) $ is smooth and has the expected dimension, namely, 
$$ \rho_H^1=3c_2+g-1.$$
\end{proposition}

\vspace{0.2cm}
Now we will outline the structure of the paper.
In Section 2, we will fix some notation and we will prove some basic facts on cohomology of line bundles on ruled surfaces and we will recall the notion of stability of a vector bundle. We will also give the general definition of the Brill-Noether locus $W_H^k(r,c_1,\dots,c_s)$ and we will recall a result about its existence and its structure as a $k$-determinantal variety. Finally, we will introduce some technical lemmas that we will use  in the next sections. 
In Section 3, we will apply the theory of walls and chambers due to Qin to describe how Brill-Noether loci change when the ample divisor crosses a wall between two chambers and we will obtain results concerning non-emptiness of $W_H^k(2;c_1,c_2)$.
In Section 4, we will push forward the study of the non-emptiness of the Brill-Noether locus $W_H^k(2;c_1,c_2)$ (see Theorems \ref{teo}, \ref{thvacio}, \ref{thc1f} and \ref{thc1fv}).
Finally, we will prove the smoothness of $W_H^1(2;f,c_2)\backslash W^2_H(2;f,c_2)$ in Proposition \ref{smooth}.

\vspace{0.3cm}
\noindent\textbf{Acknowledgements}. The second author thanks professor Marian Aprodu for his advice and comments and for his hospitality during my stay  in Bucharest.

\vspace{0.2cm}

\section{Preliminaries}
The goal of this section is to fix some notation and  prove some results concerning  the cohomology of line bundles on ruled surfaces. In addition, we will recall some definitions needed in the sequel.

A ruled surface is a surface $X$, together with a surjective map $\pi:X\rightarrow C$ to a nonsingular curve $C$ of genus $g$ such that, for every point $y\in C$, the fibre $X_y$ is isomorphic to $\mathbb{P}^1$, and such that $\pi$ admits a section. It is also defined as the projectivization $X=\mathbb{P}(\mathcal{E})$ of a normalized rank $2$ bundle $\mathcal{E}$ on $C$. Let $\mathfrak{e}$ be the divisor on $C$ corresponding to the invertible sheaf $\wedge^2\mathcal{E}$ and let us define $e:=-\deg(\mathfrak{e})$. 

Let $C_0\subseteq X$ be a section and $f$ be a fibre. We have that $\Pic(X)\cong \mathbb{Z}\bigoplus \pi^{*}\Pic(C)$, where $\mathbb{Z}$ is generated by $C_0$. Also $\Num(X)\cong \mathbb{Z}\bigoplus \mathbb{Z}$, generated by $C_0$ and $f$  satisfying $C_0^2=-e$, $C_0\cdot f=1$ and $f^2=0$.
If $\fb$ 
is a divisor on $C$, we will write $\mathfrak{b}f$ instead of $\pi^{*}\mathfrak{b}$. Thus a divisor $D$ on $X$ can be written uniquely as $D= aC_0+\mathfrak{b}f$, being $\fb\in \Pic(C)$, and any element of $\Num(X)$ can be written as $D\equiv aC_0+bf$. 
The canonical divisor $K_X$ on $X$ is in the class $-2C_0+(\mathfrak{k}+\mathfrak{e})f$, where $\mathfrak{k}$ is the canonical divisor on $C$  of degree $2g-2$.

If the genus of the curve $C$ is zero then, for any integer $e\geq0$, the ruled surface is known as a Hirzebruch surface $X_e$. In this case we have $X_e\cong \mathbb{P}(\mathcal{E})=\mathbb{P}(\mathcal{O}_{\mathbb{P}^1}\bigoplus \mathcal{O}_{\mathbb{P}^1}(-e))$.
Since $\Pic(X_e)\cong\mathbb{Z}^2$, dealing with divisors on Hirzebruch surfaces we will not make any distinction between $D=aC_0+\fb f$ and its numerical class $D\equiv aC_0+bf$. 

Going back to the general case, for any divisor $D=aC_0+\mathfrak{b} f$ on $X$ with $a\geq0$, it follows from  Lemma V.2.4, Exercises III.8.3 and
III.8.4 of  \cite{hartshorne}, that 
\begin{equation}
\label{igualdadcohom}
    \h^i(X,\mathcal{O}_X(D))=\h^i(C, (S^a\mathcal{E})(\mathfrak{b}))
\end{equation}
 where $S^a\mathcal{E}$ stands for the $a$-th symmetric power of $\mathcal{E}$.
Moreover,
\begin{equation}
    \label{desigcoh}
\h^0(C,\mathcal{O}_C(\mathfrak{b}))\leq \h^0(C, (S^a\mathcal{E})(\mathfrak{b}))\leq \sum_{i=0}^a \h^0(C,\mathcal{O}_C(\mathfrak{b}+i\fe)),  
\end{equation}
for each divisor $\mathfrak{b}$ on $C$ (see for instance \cite[Section 2]{theta}).

\begin{notation}
\rm Usually, we will write $\h^0\mathcal{O}_X(D)$ to refer to $\h^0(X,\mathcal{O}_X(D))$.
\end{notation}

\vspace{0.3cm}
From now on, we will assume that $e\geq0$. In this case,  a divisor $D= aC_0+\mathfrak{b}f$ is ample if and only if $a>0$ and $\deg({\fb})>ae$ (see \cite[Chapter V, Corollary 2.8]{hartshorne}).

\vspace{0.4cm}

For effective divisors on $X$ we have the following: 

\begin{lemma}
\label{lema}
Let $X$ be a ruled surface over a  nonsingular curve $C$ of genus $g\geq0$ and $D=aC_0+\mathfrak{b}f$ be a divisor on $X$ with $b:=\deg(\fb)$. If $D$ is effective then $a,b\geq0$.
\end{lemma}
\begin{proof}

Let us first assume that $D$ is effective.
   If $D=0$ there is nothing to prove. So, let
   let us assume that $D= aC_0+\mathfrak{b}f \neq0$.

  First of all we will see that $a\geq0$.
   Since $D$ is an effective divisor, $D\cdot H>0$ for any ample divisor $H$ on $X$. In particular, taking the ample divisor on $X$,
  $H\equiv C_0+\beta f$ with $\beta>>0$ , we have $$0<D\cdot H=-ae+a\beta+b=a(\beta-e)+b$$ what implies that $a\geq0$.
   
Now, we will see that $b\geq0$. 
Since $D$ is effective and $a\geq0$,   it follows from  (\ref{igualdadcohom}) and (\ref{desigcoh}) that $$0< \h^0\mathcal{O}_X(D)=\h^0S^{a}\mathcal{E}(\mathfrak{b})\leq \sum_{i=0}^{a}\h^0\mathcal{O}_C(\fb+i\fe).$$
In particular, since $\deg(\fe)=-e\leq0$, $\h^0\mathcal{O}_C(\fb)\neq0$ and hence $b\geq0$.
\end{proof}

\begin{lemma}
\label{lema2}
Let $X$ be a ruled surface over a nonsingular curve $C$ of genus $g\geq0$ and $D=aC_0+\fb f$ on $X$ such that $a\geq0$. Then, if $b:=\deg(\fb)$, we have
\vspace{0.4cm}

$\h^0\mathcal{O}_X(aC_0+\fb f) \leq
 \left\{ \begin{array}{lcc}
             0 &   \mbox{for}  & b<0 \\
             \\ (a+1)(b+1) &  \mbox{for} & 0\leq b\leq  g-1, \thinspace g\geq1 \\
              \\ (a+1)g &  \mbox{for} & g\leq b\leq  2g-1,  \thinspace   g\geq1   \\
             \\ (a+1)(b+1-g) &  \mbox{for}  & b \geq 2g. 
             \end{array}
   \right.
$

\end{lemma}

\begin{proof}

    Let $D=aC_0+\fb f$ be a divisor on $X$ with $a\geq0$.
    
    If $b<0$, it follows from Lemma \ref{lema} that $\h^0\mathcal{O}_X(aC_0+\fb f)=0$. Let us now consider the case $0\leq b\leq g-1$. Since $a\geq0$, by (\ref{igualdadcohom}) and (\ref{desigcoh}) we have
    $$\h^0\mathcal{O}_X(aC_0+\fb f)= \h^0S^{a}\mathcal{E}(\fb)\leq \sum_{i=0}^{a}\h^0\mathcal{O}_C(\fb+i\fe).$$
   By \cite[Chapter II, Exercise 1.15]{hartshorne}, for any divisor $\mathfrak{d}$ on $C$
   \begin{equation}
   \label{hartshorne}
       \h^0\mathcal{O}_C(\mathfrak{d})\leq \deg(\mathfrak{d})+1.
   \end{equation}
   In particular, we have
    $$\h^0\mathcal{O}_C(\fb+i\fe)\leq b-ie+1\leq b+1,$$ and thus
    $$\h^0\mathcal{O}_X(aC_0+\fb f)\leq (a+1)(b+1).$$

Let us now assume that $g\leq b\leq 2g-1$.
Since 
$a\geq0$, again by (\ref{igualdadcohom}) and (\ref{desigcoh}),
    $$\h^0\mathcal{O}_X(aC_0+\fb f)=\h^0S^{a}\mathcal{E}(\fb)\leq \sum_{i=0}^{a}\h^0\mathcal{O}_C(\fb+i\fe)
    \leq (a+1)\h^0\mathcal{O}_C(\fb)\leq (a+1)g,$$
where the last inequality follows from the fact that, since  $b\leq 2g-1$, if $\mathfrak{d}$ is a divisor of degree $2g-1$,  we get $$\h^0\mathcal{O}_C(\fb)\leq \h^0\mathcal{O}_C(\mathfrak{d})=g.$$

 Finally, let us prove the case $b\geq 2g$.
Since  $a\geq0$, by (\ref{igualdadcohom}) and (\ref{desigcoh}),
    $$\h^0\mathcal{O}_X(aC_0+\fb f)=\h^0S^{a}\mathcal{E}(\fb)\leq \sum_{i=0}^{a}\h^0\mathcal{O}_C(\fb+i\fe)
    \leq (a+1)\h^0\mathcal{O}_C(\fb)=(a+1)(b+1-g),$$
    where the last equality follows from the fact that, since $b\geq2g>2g-1 $, we have 
    ${\h^0\mathcal{O}_C(\fb)=b+1-g}$. 
\end{proof}

\vspace{0.2cm}
The following is a well-known result.
\begin{proposition}
\label{zgen}
    Let $X$ be a smooth projective surface, $D$ a divisor on $X$ and $Z\subset X$  a general 0-dimensional subscheme. If $|Z|\geq \h^0\mathcal{O}_X(D)$ then $\h^0I_Z(D)=0$.
\end{proposition}

\vspace{0.4cm}
Let $E$ be a rank 2 vector bundle on a smooth projective surface $X$. We denote by $c_i$ the Chern Classes of $E$, $c_i=c_i(E)\in \Ho^{2i}(X,E)$, for $i=1,2$.

\vspace{0.2cm}
Our aim is to study the moduli space of stable vector bundles on $X$ by means of its Brill-Noether locus. To do so, we end the section by recalling basic facts concerning stable vector bundles and Brill Noether loci.

\begin{definition}
Let $X$ be a smooth projective surface,  $H$ an ample divisor on $X$ and $E$  a vector bundle on $X$ with fixed Chern classes $c_i$ for $i=1,2$. Then $E$ is stable with respect to $H$ (or $H$-stable) if, for any subbundle $F$ of $E$, we have $$\mu_H(F):=\frac{c_1(F)\cdot H}{\rank(F)}<\frac{c_1(E)\cdot H}{\rank(E)}:=\mu_H(E).$$

Notice that the notion of stability depends on the ample divisor. Nevertheless, when there is no confusion we will simply say that $E$ is stable.
\end{definition}

\vspace{0.2cm}
Given $H$ an ample divisor on $X$, $c_1\in \Num(X)$ and $c_2\in H^4(X,\mathbb{Z})\cong \mathbb{Z}$, we denote by $M_{X,H}(r;c_1,c_2)$ the moduli space of $H$-stable rank $r$  vector bundles $E$ on $X$ with fixed Chern classes $c_i$. When there is no confusion, we will denote it by $M_H$.

\vspace{0.3cm}
For big values of $c_2$, we have a general result concerning irreducibility, smoothness and dimension of these moduli spaces. In the particular case of smooth ruled surfaces, it can be stated as follows (see for instance \cite{huybrechts}):
\begin{theorem}
\label{theomoduli}
Let $X$ be a ruled surface over a nonsingular curve $C$ of genus $g\geq0$, $H$ an ample divisor on $X$ and
${c_1, c_2 \in \Ho^{2i}(X,\mathbb{Z})}$ Chern classes. For all $c_2>> 0$, the moduli space $M_{X,H}(r;c_1, c_2)$ is a smooth, irreducible,
quasiprojective variety of dimension $2rc_2 -(r-1)c_1^2-(r^2-1)(1-g)$.
\end{theorem}

One possible way to study the moduli space of stable vector bundles is by describing the Brill-Noether locus $W_H^k:=W_H^k(r;c_1,c_2)$ in $M_H$ which parametrizes stable rank $r$ vector bundles having at least $k$  independent sections, that is $$\Supp(W_H^k(r;c_1,c_2))=\{E\in M_{X,H}(r;c_1,c_2)| \thinspace \h^0E\geq k\}.$$

\vspace{0.2cm}

The Brill-Noether locus has the structure of a $k$-determinantal variety $W_H^k(r; c_1,c_2)$, and so we can talk about its expected dimension and singular locus. This is given by the following result (see \cite[Theorem 2.3]{Costa}).
\begin{theorem}

\label{theoestructura}
     Let $X$ be a smooth projective surface and consider the
moduli space $ M_{X,H}(r; c_1,c_2)$ of rank $r$, $H$-stable vector bundles $E$ on $X$ with
fixed Chern classes $ c_i$. Assume that for any $E \in M_H$, $\h^iE = 0$ for $ i \geq 2$.
Then, for any $k\geq0$, there exists a determinantal variety $W_H^k(r; c_1, c_2)$
 such that
$$\Supp(W_H^k(r; c_1, c_2)) = \{E\in M_H| \thinspace \h^0E\geq k\}.$$
Moreover, each non-empty irreducible component of $W_H^k(r;c_1, c_2)$ has dimension at
least $$ \dim(M_H)-k(k-\chi(r; c_1, c_2)),$$
and
$$W_H^{k+1}(r;c_1, c_2)\subset \Sing(W_H^k(r;c_1, c_2))$$
whenever $W_H^{k+1}(r;c_1, c_2)\neq M_H$.
\end{theorem}

\vspace{0.2cm}

\begin{remark}
    By definition, the different Brill-Noether loci form a filtration of the moduli space $M_{X,H}(r;c_1,c_2)$,
    $$M_H\supseteq W_H^1(r;c_1,c_2)\supseteq W_H^2(r;c_1,c_2)\supseteq \cdots \supseteq W_H^{k}(r;c_1,c_2)\supseteq W_H^{k+1}(r;c_1,c_2)\supseteq \cdots$$
\end{remark}

\vspace{0.2cm}

\begin{definition}
    The expected dimension of the Brill-Noether locus is defined as $$\rho_H^k:=\rho_H^k(r;c_1,c_2)=\dim M_{X,H}(r;c_1,c_2)-k(k-\chi(r;c_1,c_2)).$$ By the Riemann-Roch theorem,  $\chi(r;c_1,c_2)$ can be expressed as 
     $$\chi(r;c_1,c_2)=r(1+p_a(X))-\frac{c_1K_X}{2}+\frac{c_1^2}{2}-c_2.$$
    
    In particular, if $X$ is a ruled surface over a smooth curve $C$ of genus $g$, then
    $$\chi(r;c_1,c_2)=r(1-g)-\frac{c_1\cdot [-2C_0+(\mathfrak{k}+\mathfrak{e})f]}{2}+\frac{c_1^2}{2}-c_2.$$
\end{definition}

\begin{remark} \rm
\rm Under the assuption $c_2>>0$, the moduli space has the dimension given by Theorem \ref{theomoduli} and hence, if $X$ is a smooth ruled surface over a nonsingular curve of genus $g$, $$\rho_H^k:=\rho_H^k(r;c_1,c_2)=2rc_2-(r-1)c_1^2-(r^2-1)(1-g)-k(k-\chi(r;c_1,c_2)).$$
\end{remark}

\vspace{0.2cm}
\begin{lemma}
\label{lemacondicion}
Let $E$ be a rank $r$, $H$-stable  vector bundle on a ruled surface $X$ with 

$(c_1(E)-rK_X)\cdot H\geq0$. Then, $\h^2E=0$.
\end{lemma}
\begin{proof}
    Let us suppose $\h^2E>0$. This means, by Serre duality, that $\h^0E^{*}(K_X)>0$, what implies $\mathcal{O}_X(-K_X)\hookrightarrow E^*$. 
    Since $E$ is $H$-stable, the same is true for $E^{*}$, and thus we have $$\mu_H(\mathcal{O}_X(-K_X))<\mu_H(E^{*}),$$ which is equivalent to $$(-rK_X-c_1(E^{*}))\cdot H=(-rK_X+c_1(E))\cdot H<0,$$ what is a contradiction. Hence $\h^2E=0$.
\end{proof}

From now on, we will assume that $(c_1(E)-rK_X)\cdot H\geq 0 $ for all $E\in M_{X,H}(r;c_1,c_2)$ and therefore, by Theorem \ref{theoestructura}  and applying  Lemma \ref{lemacondicion}, the existence of the Brill-Noether locus $W_H^k(r;c_1,c_2)$ is guaranteed.

\vspace{0.1cm}
While dealing with the non-emptiness  of the Brill-Noether locus $W_H^k(r;c_1,c_2)$, notice that if $c_1\cdot H\leq 0$, then $W_H^k(r;c_1,c_2)=\emptyset$ for all $k\geq1$. In fact, if we have $E\in W_H^k(r;c_1,c_2)$ then $h^0E\neq0$ and thus $\mathcal{O}_X\hookrightarrow E$. By stability of $E$ and because $c_1\cdot H\leq0$, we have $0=\mu_H(\mathcal{O}_X)<\mu_H(E)\leq0$, what is a contradiction. Hence $W_H^1(r;c_1,c_2)=\emptyset$ whenever $c_1\cdot H\leq0$

 In particular, $W_H^1(r;0,c_2)=\emptyset$. Therefore we will assume that $c_1$ is effective and non zero, and thus $c_1\cdot H>0$ for any ample divisor $H$ . Furthermore, in the case of a rank $2$ vector bundle $E$, we have that $E$ is $H$-stable if and only if $E\otimes L$ is $H$-stable, for any line bundle $L$ on $X$, so it is natural to focus the attention on the cases $c_1\in\{C_0,C_0+f,f\}$.

\vspace{0.2cm}

\section{Walls and chambers structure in Brill-Noether loci}

We start the section recalling the theory of walls and chambers introduced by Qin, in the context of smooth surfaces (see for instance \cite{Qin1}). The main idea is that if $X$ is a smooth surface, the ample cone of $X$, $C_X$, has a structure of walls and chambers such that the moduli space $M_{H}(2;c_1,c_2)$ only depends on the chamber of $H$.
\vspace{0.2cm}

To introduce the main results of this theory, let us remember some basic definitions.

\begin{definition} \begin{itemize}
        \item[(i)] Let $\zeta \in \Num(X)\otimes \mathbb{R}$. We define $$W^{\zeta}=C_X\cap \{x\in \Num(X)| x\cdot\zeta =0\}.$$
        \item[(ii)] Define $\mathcal{W}(c_1,c_2)$ as the set whose elements consist of $W^{\zeta}$, where $\zeta$ is the numerical equivalence class of a divisor $D$ on $X$ such that $D+c_1$ is divisible by $2$ in $\Pic(X)$, $-(4c_2-c_{1}^2)\leq D^2<0$ and $|Z|=c_2+\frac{D^2-c_1^2}{4}$ for some locally complete intersection codimension-two cycle $Z$ in $X$.
        \item[(iii)] A wall of type $(c_1, c_2)$ is an element in $\mathcal{W}(c_1, c_2)$. A chamber of type $(c_1, c_2)$ is a
connected component of $C_X \backslash W(c_1, c_2)$. A $Z$-chamber of type $(c_1, c_2)$ is the intersection
of $\Num(X)$ with some chamber of type $(c_1, c_2)$.
\item[(iv)]  A face of type $(c_1, c_2)$ is $\mathcal{F} = W^{\zeta}\cap \mathcal{C}$, where $W^{\zeta}$
is a wall of type $(c_1, c_2)$ and $\mathcal{C}$ is a chamber of type $(c_1,c_2)$.

    \end{itemize}
\end{definition}

\begin{definition}
    Let $\zeta$ be a numerical equivalence class defining a wall of type $(c_1,c_2)$. We define $E_{\zeta}(c_1,c_2)$
 as the set of rank $2$ vector bundles $E$ on $X$ given by a non-trivial extension of type 
 $$0\rightarrow \mathcal{O}_X(D)\rightarrow E\rightarrow I_Z(c_1-D)\rightarrow 0$$
where $D$ is a divisor with $2D- c_1 \equiv \zeta$ and $Z$ is a locally complete intersection 0-cycle of
length $c_2+\frac{\zeta^2-c_1^2}{4}$.
\end{definition}

Qin proved in \cite[Proposition 2.2.5]{Qin1} and \cite[Remark 2.2.6]{Qin1} that the moduli space $M_{H}(2;c_1,c_2)$ only depends on the chamber of $H$, so he set $M_{\mathcal{C}}(c_1,c_2)$ to be $M_{H}(2;c_1,c_2)$ for some $H\in \mathcal{C}$. 

Once we know that $M_H(2;c_1,c_2)$ only depends on the chamber of $H$, the natural question that arises is how the moduli space changes when the polarization $H$ crosses a wall $W^{\zeta}$ between two chambers. The phenomena that occurs is summarized in the following result.

\begin{theorem}(\cite[Proposition 1.3.1]{Qin1})
\label{thwall1}
   Let $\mathcal{C}$ be a chamber and $\mathcal{F}$ be one of its faces. Then, as sets, 
    $$M_{\mathcal{C}}(c_1,c_2)=M_{\mathcal{F}}(c_1,c_2)\amalg (\amalg_{\zeta} E^k_{\zeta}(c_1,c_2))$$
    where $\zeta$ satisfies $\zeta\cdot H<0$ for some $H\in \mathcal{C}$, and runs over all numerical equivalence classes which define the wall containing $\mathcal{F}$.
\end{theorem}

\noindent As a consequence 

\begin{theorem}(\cite[Theorem 1.3.3]{Qin1})
\label{thwall2}
    Let $\mathcal{C}_1$ and $\mathcal{C}_2$ be two adjancent chambers of type $(c_1,c_2)$ sharing a common wall $W^{\zeta}$ such that $\mathcal{C}_1$ lies above $\mathcal{C}_2$. Then, as sets,
    $$M_{\mathcal{C}_2}(c_1,c_2)=(M_{\mathcal{C}_1}(c_1,c_2)\backslash\amalg_{\zeta}E_{-\zeta}(c_1,c_2))\amalg(\amalg_{\zeta}E_{\zeta}(c_1,c_2))$$ where $\zeta$ runs over all normalized numerical equivalence classes (that is $\zeta\cdot f>0$) which represents the wall $W^{\zeta}$.
\end{theorem}

It is natural to ask ourselves what happens at the level of Brill-Noether loci. To this end, let us define
 $$W_{\mathcal{C}}^k(c_1,c_2):=\{E\in \mathcal{M}_{\mathcal{C}}(c_1,c_2) |  h^0E\geq k\}$$ and $$E_{\zeta}^k(c_1,c_2):=\{E\in E_{\zeta}| h^0E\geq k\}.$$

\vspace{0.1cm}
As a consequence of Theorem \ref{thwall1} and Theorem \ref{thwall2}  we can deduce

\begin{proposition}
\label{propwall}
\begin{itemize}
    \item[(a)]     Let $\mathcal{C}$ be a chamber and $\mathcal{F}$ one of its faces. Then, as sets, $$W^k_{\mathcal{C}}(c_1,c_2)=W^k_{\mathcal{F}}(c_1,c_2)\amalg (\amalg_{\zeta} E^k_{\zeta}(c_1,c_2))$$
    where $\zeta$ satisfies $\zeta\cdot H<0$ for some $H\in \mathcal{C}$, and runs over all numerical equivalence classes which define the wall containing $\mathcal{F}$.
    \item[(b)] Let $\mathcal{C}_1$ and $\mathcal{C}_2$ be two adjancent chambers of type $(c_1,c_2)$ sharing a common wall $W^{\zeta}$ such that $\mathcal{C}_1$ lies above $\mathcal{C}_2$. Then, as sets,
    $$W^k_{\mathcal{C}_2}(c_1,c_2)= (W^k_{\mathcal{C}_1}(c_1,c_2)\backslash\amalg_{\zeta}E^k_{-\zeta}(c_1,c_2))\amalg(\amalg_{\zeta}E^k_{\zeta}(c_1,c_2))$$
\end{itemize}

\end{proposition}

We will use this wall crossing to prove the following result.

\begin{theorem}
    \label{teowall}
    Let $X$ be a ruled surface over a nonsingular curve $C$ of genus $g\geq0$, $m\in\{0,1\}$ and $c_2>>0$ an integer.
    Let us consider the family of numerical equivalence classes $$\zeta_{b}\equiv (2b-m)f-C_0$$ with $0<b=k-1+g<c_2$ and let $\mathcal{C}_b$ be the chamber such that $W^{\zeta_b}\cap \text{Closure}(\mathcal{C}_b)\neq \emptyset$ and $\zeta_b \cdot H<0$ for all $H\in \mathcal{C}_b$. Then, for any ample divisor $H\in \mathcal{C}_b$, $W_H^k(2;C_0+\fm f, c_2)\neq \emptyset$ whenever $M_H\neq \emptyset$.
    
\end{theorem}

\begin{proof}
 First of all, since we have the filtration $W_H^{l+1}\subseteq W_H^l\subseteq \cdots \subseteq W_H^1\subseteq M_H$, we can assume $k\geq \max\{2,g\}$ and thus $b\geq \max\{g+1,2g-1\}$.

 Notice that $\zeta_{b}\equiv (2b-m)f-C_0$ defines a non-empty wall of type $(C_0+\fm f,c_2)$. In fact,
       $$(\zeta_{b})^2= ((2b-m)f-C_0)^2= -4b-e+2m.$$ Since $m\in\{0,1\}$ and by the lower and upper bounds of $b$, we have $$-4c_2-e\leq -4b-e+2m <0.$$  On the other hand,
         $\zeta_{b}+C_0+\fm f = 2\fb f $ is divisible by $2$ in $\Pic(X)$ and
         $H\equiv C_0+ (2b+e-m)f$ is an ample divisor on $X$ such that $\zeta_b\cdot H=0$.
         Hence, $\zeta_b$ defines a non-empty wall of type $(C_0+\fm f, c_2)$.

Let us consider $E_{\zeta_b}(C_0+\fm f, c_2)$, which is by definition the family of rank 2 vector bundles given by non-trivial extensions of type
\begin{equation}
    \label{extensionwall}
 0\rightarrow \mathcal{O}_X(\fb f)\rightarrow E\rightarrow I_Z(C_0+(\fm-\fb)f)\rightarrow 0,
\end{equation}
where $\fb f$ is a divisor such that $\zeta_b\equiv (2b-m)f-C_0$ with $b=\deg(\fb)$ and $Z$ a  0-dimensional subscheme of length $|Z|=c_2-b$.

\vspace{0.2cm}
\noindent{\bf{Claim}}: $E_{\zeta_b}(C_0+\fm f, c_2)$ is non-empty.

\noindent {\bf{Proof of the Claim}}:
It is enough to see that $\Ext^1(I_Z(C_0+(\fm-\fb)f),\mathcal{O}_{X}(\fb f))\neq0$.
We have that $\ext^1(I_Z(C_0+(\fm-\fb)f),\mathcal{O}_{X}(\fb f))=h^1I_Z(C_0+(\fm-2\fb)f+K_X)$.
If we consider the short exact sequence
\begin{equation}
    0\rightarrow I_Z\rightarrow \mathcal{O}_X\rightarrow \mathcal{O}_Z\rightarrow 0,
    \label{sucz}
\end{equation}
we tensor it by $\mathcal{O}_X(C_0+(\fm -\fb)f+K_X)$ and we take cohomology,
we get $$\h^0I_Z(C_0+(\fm-2\fb)f+K_X)\leq \h^0\mathcal{O}_X(C_0+(\fm-2\fb)f+K_X))=0,$$
where the last equality follows from Lemma \ref{lema}.

On the other hand,
if we consider again the exact sequence  (\ref{sucz})
and we take cohomology, we get 
$$\h^2I_Z(C_0+(\fm-2\fb)f+K_X)=\h^2\mathcal{O}_X(C_0+(\fm-2\fb)f+K_X)$$ and, by duality and Lemma \ref{lema},  $$\h^2\mathcal{O}_X(C_0+(\fm-2\fb)f+K_X)=\h^0\mathcal{O}_X(-C_0-(\fm-2\fb)f)=0.$$
Therefore,  $$\h^1I_Z(C_0+(\fm-2\fb)f+K_X)=-\chi I_Z(C_0+(\fm-2\fb)f+K_X).$$  Considering the exact sequence  (\ref{sucz}) and the Riemann-Roch Theorem,
$$-\chi I_Z(C_0+(\fm-2\fb)f+K_X)=|Z|-\chi \mathcal{O}_X(C_0+(\fm-2\fb)f+K_X)=|Z|.$$ 

\vspace{0.3cm}
Putting altogether we have $$\ext^1(I_Z(C_0+(\fm-\fb)f),\mathcal{O}_{X}(\fb f))=|Z|=c_2-b> 0,$$
where the last inequality follows from the upper bound of $b$.
    
Therefore we can take a rank 2 vector bundle $E\in E_{\zeta_b}(C_0+\fm f, c_2)$. It follows from Theorem \ref{thwall1} that $E_{\zeta_b}(C_0+\fm f)\subseteq M_{\mathcal{C}_b}(C_0+\fm f, c_2)$ and thus $E$ is $H$-stable for any $H\in \mathcal{C}_b$.

Finally, since $E$ is given by the exact sequence (\ref{extensionwall}), taking cohomology we get $$\h^0E\geq \h^0\mathcal{O}_X(\fb f)=\h^0\mathcal{O}_C(\fb)=\chi \mathcal{O}_C(\fb)=b+1-g=k,$$ where the second equality follows from the fact that, since $k\geq \max\{1,g\}$, $\h^0\mathcal{O}_C(\fb)=0$.

Thus, we have seen that $\emptyset\neq E_{\zeta_b}(C_0+\fm f,c_2)\subseteq W_{\mathcal{C}_b}^k(2;C_0+\fm f, c_2)$ and thus $W_H^k(2;C_0+\fm f, c_2)\neq \emptyset$ for all $H\in \mathcal{C}_b$.
\end{proof}

    Notice that with the above notations, if $H\equiv \alpha C_0+\beta f$ the condition $\zeta_b\cdot H<0$ is equivalent to $$k<\frac{1}{2\alpha}[\beta - \alpha(e-m+2g-2)].$$
    Hence,
    
    \vspace{0.2cm}
    \begin{corollary}
     Let $X$ be a ruled surface over a nonsingular curve $C$ of genus $g\geq0$, $m\in\{0,1\}$ and $c_2>>0$ an integer.  Let us consider the family of numerical equivalence classes $$\zeta_{b}\equiv (2b-m)f-C_0$$ with $0<b=k-1+g<c_2$.
     Assume $M_H\neq\emptyset$.
     Then
          $W_H^k(2;C_0+\fm f, c_2)\neq \emptyset$ for $$1\leq k<\frac{1}{2\alpha}[\beta + \alpha(e-m+2g-2)]$$ and $H\in \mathcal{C}$ being $\mathcal{C}$ the chamber such that $W^{\zeta_b}\cap \text{Closure}(\mathcal{C})\neq \emptyset$. 
    \end{corollary}

\section{Emptiness and non-emptiness of Brill-Noether loci }

The goal of this section is to determine lower bounds of $k$, depending on an ample divisor $H$, that garantize that the Brill-Noether locus $W_H^k(2;c_1,c_2)$ is non-empty. Moreover, we will see that these bounds are sharp. In contrast with the above section, we will obtain non-emptiness results by explicitly constructing families of $H$-stable rank 2 vector bundles for, in many cases, almost all ample divisors $H$ and with values of the first Chern classes $c_1$ in $\{C_0,C_0+ f, f\}$.

\vspace{0.2cm}
We will start with the case $c_1=C_0+\mathfrak{m}f$ with $\deg(\mathfrak{m})=m\in\{0,1\}$.  

\vspace{0.1cm}

\begin{theorem}
\label{teo}
Let $X$ be a ruled surface over a nonsingular curve $C$ of genus $g\geq0$, $m\in\{0,1\}$, $c_2>>0$ an integer and $H\equiv\alpha C_0+\beta f$ an ample divisor on $X$ with 
$$\alpha (e+m)<\beta.$$
Then, for any $k$ in the range $$\max\{1,g\}\leq  k<\frac{1}{2\alpha}[\beta -\alpha (e- m+2g-2 )],$$ the Brill-Noether locus $W_H^k(2;C_0+\fm f,c_2)\neq\emptyset$ whenever $M_H\neq \emptyset$.
\end{theorem}
\vspace{0.1cm}

\begin{proof}
First of all notice that one can extend \cite[Lemma 1.10]{Qin2} to any ample divisor $H\equiv \alpha C_0+ \beta f$ and prove that if $ M_H(2;C_0+\fm f, c_2)\neq \emptyset$ then $\beta <\alpha(2c_2+e-m)$.  Hence we can assume $\beta <\alpha(2c_2+e-m)$.
Let us consider the family $\mathcal{G}$ of rank two vector bundles $E$ on $X$ given by a non-trivial extension 
\begin{equation}
    \label{extension1}
 0\rightarrow \mathcal{O}_X(\fb f)\rightarrow E\rightarrow I_Z(C_0+(\fm-\fb)f)\rightarrow 0,
\end{equation}
where $Z$ is a generic 0-dimensional subscheme of length $|Z|=c_2-b$ with
 $$b=\deg(\mathfrak{b})=k-1+g$$ and   $$\max\{1,g\}\leq  k<\frac{1}{2\alpha}[\beta -\alpha (e- m+2g-2 )].$$

\vspace{0.4cm}
\noindent{\bf{Claim}}: $\mathcal{G}$ is non-empty.

\noindent{\bf{Proof of the Claim}}:
Notice that since $\h^0I_Z(C_0+(\fm-2\fb)f+K_X)=0=\h^2I_Z(C_0+(\fm-2\fb)f+K_X)$, using Riemann-Roch Theorem and the upper bound of $b$ we deduce that
$$\ext^1(I_Z(C_0+(\fm-\fb)f),\mathcal{O}_{X}(\fb f))=|Z|=c_2-b> 0,$$

what implies $\Ext^1(I_Z(C_0+(\fm-\fb)f),\mathcal{O}_{X}(\fb f))\neq0$.

\vspace{0.3cm}

Hence any non-zero element in $\Ext^1(I_Z(C_0+(\fm-\fb)f),\mathcal{O}_{X}(\fb f))$ defines a non-trivial extension of type (\ref{extension1}).
Notice that by the Cayley-Bacharach property \cite[Theorem 12]{friedman},
$E$ is indeed a rank two vector bundle, with $c_1(E)=C_0+\fm f$ and $c_2(E)=c_2>>0$.
Hence, $\mathcal{G}$ is non-empty.

\vspace{0.4cm}
On the other hand, if we consider the exact sequence (\ref{extension1}) we get
 $$\h^0E\geq \h^0\mathcal{O}_X(\fb f)=\h^0\mathcal{O}_C(\fb)$$
 and, since $k\geq \max\{1,g\}$ we get $\h^1\mathcal{O}_C(\fb)=0$ and thus $\h^0\mathcal{O}_C(\fb)=\chi\mathcal{O}_C(\fb)=b+1-g=k$.

\vspace{0.4cm}
Now, let us prove that any rank two vector bundle $E$ in $\mathcal{G}$ is stable with respect to the polarization $H\equiv\alpha C_0+\beta f$ with $\alpha (e+m)<\beta<\alpha(2c_2+e-m)$.

To this end, since $E$ is a rank two vector bundle on $X$, we have to check that for any line subbundle $\mathcal{O}_X(G)\hookrightarrow E$ we have $$\mu_H(\mathcal{O}_X(G))=G\cdot H < \mu_H(E)=\frac{c_1(E)\cdot H}{2}.$$

By construction, $E$ is given by a non-trivial extension of type (\ref{extension1}). Therefore we have two possibilities:  
$\mathcal{O}_X(G)\hookrightarrow\mathcal{O}_X(\fb f) $ or  $\mathcal{O}_X(G)\hookrightarrow I_Z(C_0+(\fm-\fb)f)$.

\vspace{0.4cm}
{\bf Case 1}
Assume that $\mathcal{O}_X(G)\hookrightarrow\mathcal{O}_X(\fb f)$. In this case, since $$k<\frac{1}{2\alpha}[\beta -\alpha (e- m+2g-2 )],$$
we get
$$\mu_H(\mathcal{O}_X(G))\leq (\fb f)\cdot H=\alpha b=\alpha k+\alpha(-1+g)$$
$$<\frac{\alpha}{2\alpha}[\beta-\alpha(e-m+2g-2)]+\alpha(-1+g)=\frac{1}{2}(\beta -\alpha e+\alpha m)=\mu_H(E).$$

\vspace{0.4cm}
{\bf Case 2}
Assume that $\mathcal{O}_X(G)\hookrightarrow I_Z(C_0+(\fm-\fb)f)$, where $G=s C_0+\ft f$ with $t=\deg(\mathfrak{t})$.
In this case, $$C_0+(\fm-\fb)f-G=(1-s)C_0+(\fm-\fb-\ft)f$$ is an effective divisor, which by Lemma \ref{lema} implies that $s\leq1$ and $t\leq m-b$.

Assume that $s\leq0$.
Since  $t\leq m-b$, we have
$$\mu_H(\mathcal{O}_X(G))= s(\beta-\alpha e)+t\alpha \leq \alpha(m-b)<\frac{1}{2}(\beta-\alpha e+\alpha m)=\mu_H(E),$$

where the last inequality follows from the fact that $b\geq0$ and $\beta>\alpha(e+m)$.

\vspace{0.1cm}

Finally assume that $s=1$. In this case $G=C_0+\ft f$ and since $$\mathcal{O}_X(G)\hookrightarrow I_Z(C_0+(\fm-\fb)f),$$ we have $\h^0I_Z(C_0+(\fm-\fb)f-G)=\h^0I_Z((\fm-\fb-\ft)f)\neq0$. Since $Z$ is a generic $0$-dimensional subscheme, we must have $$|Z|<\h^0\mathcal{O}_X((\fm-\fb-\ft)f)=\h^0\mathcal{O}_C(\fm-\fb-\ft)\leq m-b-t+1,$$
where the last inequality is given by  (\ref{hartshorne}).

Thus, since $|Z|=c_2-b$,
\begin{equation}
t\leq m-c_2. 
\label{cota}
\end{equation}

\vspace{0.1cm}
Now, let us check  that $\mu_H(\mathcal{O}_X(C_0+\ft f))<\mu_H(E)$. 
Notice that ${\mu_H(\mathcal{O}_X(C_0+\ft f))<\mu_H(E)}$ if and only if $$-\alpha e+\beta+t\alpha <\frac{1}{2}(\beta-\alpha e+\alpha m)$$ what is equivalent to $$-\alpha e+\beta-\alpha m<-2t\alpha.$$
This inequality holds since, by the upper bound of $\beta$, and according to (\ref{cota}), $$-2t\alpha\geq-2\alpha(m-c_2)>-\alpha e+\beta-\alpha m.$$
Therefore $\mu_H(\mathcal{O}_X(C_0+\ft f))<\mu_H(E)$ and putting altogether we get that $E$ is $H$-stable.

\vspace{0.4cm}
We have proved that any $E$ in $\mathcal{G}$ is a rank two $H$-stable vector bundle with ${\h^0E\geq k}$, what implies that $\mathcal{G}\hookrightarrow W_H^k(2;C_0+\fm f,c_2)$ and hence $W_H^k(2;C_0+\fm f,c_2)\neq \emptyset$.
 \end{proof}

\begin{remark}

Notice that the condition $\beta >\alpha(e+m)$ in Theorem \ref{teo} is not restrictive. In fact, it is verified for almost all ample divisors, since $m\in \{0,1\}$ and $\beta >\alpha e$ for any ample divisor $H\equiv \alpha C_0+ \beta f$.
   
\end{remark}

\vspace{0.2cm}
 \begin{corollary}
 \label{coremp}
 Let $X$ be a ruled surface over a nonsingular curve $C$ of genus $g\geq0$, $m\in\{0,1\}$, $c_2>>0$ an integer and $H\equiv\alpha C_0+\beta f$ an ample divisor on $X$   with $$\alpha (e+m)<\beta.$$
 Then, for any 
 $$1\leq k< \frac{1}{2\alpha}[\beta -\alpha(e-m+2g-2 )],$$ $W_H^k(2;C_0+\mathfrak{m}f,c_2)\neq\emptyset$ whenever $M_H\neq \emptyset$.
 \end{corollary}

\begin{proof}
    If $\max\{1,g\}\leq k< \frac{1}{2}\alpha[\beta -\alpha(e+m+2g-2 )]$, the result follows from Theorem \ref{teo}. Finally, since we have the filtration $W_H^1\supseteq W_H^2\supseteq \cdots\supseteq W_H^q$ with $q=\max\{1,g\}$, the fact that $W_H^{q}\neq\emptyset$ implies that $W_H^l\neq\emptyset$ for all  $1\leq l\leq q $.
\end{proof}

\vspace{0.2cm}

\vspace{0.2cm}

\begin{remark}

\noindent    
    Notice that the expected dimension of the Brill-Noether locus $W_H^k(2;C_0+\fm f,c_2)$ is 
    $$\rho_H^k(2;C_0+\fm f,c_2)=(4-k)c_2+3k-3+3(1-k)g+(1-k)e+2(k-1)m-k^2.$$
    
   \noindent  
    Hence, for $k\geq 4$ we have $\rho_H^k<0$, but for $$1\leq g\leq k < \frac{1}{2\alpha}[\beta -\alpha(e-m+2g-2 )]$$ we have seen that $W_H^k\neq \emptyset$. So, we have an example of a non-empty Brill-Noether locus with negative expected dimension and therefore it is clear that, in general, $\rho_H^k<0$ does not imply that $ W_H^k=\emptyset$.

\end{remark}

 \vspace{0.2cm}

Our next goal is to see that the above bound is sharp, in the sense that if $$k\geq \frac{1}{2\alpha}[\beta-\alpha(e-m+2g-2)],$$ then $W_H^k(2;C_0+\fm f,c_2)=\emptyset$.

\vspace{0.3cm}
\begin{theorem}
\label{thvacio}
Let $X$ be a ruled surface over a nonsingular curve $C$ of genus $g\geq0$, $m\in\{0,1\}$, $c_2>>0$ an integer and $H\equiv\alpha C_0+\beta f$ be an ample divisor on $X$ such that $$\alpha(8+e-m+2g)<\beta.$$

\noindent Then $W_H^k(2;C_0+\fm f, c_2)=\emptyset$ for all $k\geq \frac{1}{2\alpha}[\beta-\alpha(e-m+2g-2)]$.

\end{theorem}

\begin{proof}

Let us assume that $W_H^k(2;C_0+\fm f,c_2)\neq \emptyset$.
 Since $\h^0E\neq0$, we can take a non-zero  section $s$ of $E$. We denote by $Y$ its scheme of zeros and by $D= aC_0+\fb f$ the maximal effective divisor contained in $Y$. Then $s$ can be regarded as a section of $E(-D)$ and its scheme of zeros has codimension greater or equal than two. Thus, we have a short exact sequence
\begin{equation}
    0\rightarrow \mathcal{O}_X(D)\rightarrow E\rightarrow I_Z(C_0+\fm f-D)\rightarrow0
    \label{extension2}
\end{equation}
where $Z$ is a locally complete intersection $0$-cycle of length $|Z|=c_2-D(C_0+\fm f-D)$.

\vspace{0.2cm}
Since $D$ is effective, by Lemma \ref{lema}, $a\geq 0$ and $b:=\deg(\fb)\geq 0$. 

If $D=0$, considering the exact sequence  (\ref{extension2}) and taking cohomology, we have that $$k=\h^0E\leq \h^0\mathcal{O}_X+ \h^0I_Z(C_0+\fm f)=1+\h^0I_Z(C_0+\fm f).$$
By the short exact sequence 
\begin{equation}
\label{sucz2}
    0\rightarrow I_Z\rightarrow\mathcal{O}_X  \rightarrow\mathcal{O}_Z\rightarrow0
\end{equation}
 and  by Lemma \ref{lema2}, we have $$\h^0I_Z(C_0+\fm f)\leq \h^0\mathcal{O}_X(C_0+\fm f)\leq 2(m+1)-e\leq 4.$$ This implies $k\leq 5$, but  this contradicts the fact that $k\geq \frac{1}{2\alpha}[\beta-\alpha(e-m+2g-2)]>5$.  
 Hence $D\neq 0$. 

Now let us see that if $D\neq0$, then  $D=\fb f$ with $b>0$.

Since $E$ is a rank two $H$-stable vector bundle, $\mu_H(\mathcal{O}_X(D))<\mu_H(E)$ and this is equivalent to 
\begin{equation}
\label{eqest}
    \alpha (2b-m)<(1-2a)(\beta -\alpha e) .
\end{equation}

Let us see that $a=0$. If $a>0$, since $\beta-\alpha e> \alpha(8+2g-m)$ and considering the inequality in (\ref{eqest}), we get $2b+m< -\alpha(8+2g-m)$ which implies $b<0$,  a contradiction. Thus $a=0$ and hence $D=\fb f$. 
In addition, since $D$ is a non-zero effective divisor, we get $b>0$.
Therefore, by (\ref{extension2}), $E$ sits in the exact sequence 
\begin{equation}
    \label{sucesion2}
    0\rightarrow \mathcal{O}_X(\fb f)\rightarrow E\rightarrow I_Z(C_0+(\fm-\fb)f)\rightarrow0.
\end{equation}
Considering the exact sequence  (\ref{sucesion2}) and taking cohomology, we get $$k\leq \h^0E\leq \h^0\mathcal{O}_X(\fb f)+\h^0I_Z(C_0+(\fm-\fb)f).$$
On the one hand, following Lemma \ref{lema2}   , $\h^0\mathcal{O}_X(\fb f)=\h^0\mathcal{O}_C(\fb)\leq b+1$.
If we consider the exact sequence  (\ref{sucz2}) and we take cohomology, we get $$\h^0I_Z(C_0+(\fm-\fb)f)\leq \h^0\mathcal{O}_X(C_0+(\fm-\fb)f).$$ Since by Lemma \ref{lema2}, $$\h^0\mathcal{O}_X(C_0+(\fm-\fb)f)\leq 2(m-b+1)-e,$$ we get $$\h^0I_Z(C_0+(\fm-\fb)f)\leq 2(m-b+1)-e.$$
Putting altogether, $$k\leq \h^0E\leq b+1+2(m-b+1)-e=2m+3-b-e,$$
what is equivalent to $b\leq 2m+3-e-k$. Since $k\geq\frac{1}{2\alpha}[\beta-\alpha(e-m+2g-2)]$ and 
 $\beta >\alpha(e-m+2g+8)$,  we have  $$b\leq 2m+3-e-k<2m+3-e-\frac{1}{2\alpha}[\beta-\alpha(e-m+2g-2)]<2m+3-e-5\leq0,$$ what is a contradiction since $b>0$. 
Hence $W_H^k(2;C_0+\fm f,c_2)=\emptyset$ for all $k\geq \frac{1}{2\alpha}[\beta-\alpha(e-m+2g-2)]$.
\end{proof}

\begin{remark}
   If we consider the particular case $g=1$ in  Theorem \ref{thvacio}, the result is also true if we expand the family of ample line bundles in $X$, to the one of $H\equiv \alpha C_0+\beta f$ with $$\alpha(6+e-m)<\beta.$$

\end{remark}

\vspace{0.2cm}
As a consequence of the above results we have the following equivalence.

\begin{corollary}
    Let $X$ be a ruled surface over a nonsingular curve $C$ of genus $g\geq0$, $m\in\{0,1\}$, $c_2>>0$ an integer and $H\equiv\alpha C_0+\beta f$ an ample divisor on $X$ such that $$\alpha(8+e-m+2g)<\beta.$$
Assume that $M_H\neq\emptyset$.
Then, $W_H^k(2;C_0+\fm f, c_2)\neq\emptyset$ if and only if $$1\leq k<\frac{1}{2\alpha}[\beta-\alpha(e-m+2g-2)].$$

\end{corollary}

\begin{proof}
    It follows from Corollary \ref{coremp} and Theorem \ref{thvacio}.
\end{proof}

\vspace{0.2cm}

\vspace{0.3cm}

Let us now turn our attention to the case $c_1=f$ and see under which conditions $W_H^k(2;f,c_2)$ is non-empty.
The first result is about non-emptiness.

\begin{theorem}
\label{thc1f}
Let $X$ be a ruled surface over a nonsingular curve $C$ of genus $g\geq0$,  $c_2>>0$ an integer  and $H\equiv C_0+\beta f$ an ample divisor on $X$.
Then  $W_H^1(2;f,c_2)\neq\emptyset$ whenever $M_H\neq\emptyset$.
\end{theorem}

\begin{proof}

Let us consider the family $\mathcal{G}$ of rank two vector bundles $E$ on $X$ given by a non-trivial extension 
\begin{equation}
    \label{extension3}
 0\rightarrow \mathcal{O}_X\rightarrow E\rightarrow I_Z(f)\rightarrow 0,
\end{equation}
where $Z$ is a 0-dimensional subscheme of length $|Z|=c_2$. 
\vspace{0.4cm}

\noindent{\bf{Claim}}: $\mathcal{G}$ is non-empty.

\noindent{\bf{Proof of the Claim}}:
First of all let us see that $\Ext^1(I_Z(f),\mathcal{O}_{X})\neq0$.
By Serre's duality, $\ext^1(I_Z(f),\mathcal{O}_{X})=\h^1I_Z(f+K_X)$.

From the short exact sequence (\ref{sucz2}),
by Lemma \ref{lema} and Serre's duality,
  we have $$\h^0I_Z(f+K_X)\leq \h^0\mathcal{O}_X(f+K_X)=0$$
  and also  $$\h^2I_Z(f+K_X)=\h^2\mathcal{O}_X(f+K_X)=\h^0\mathcal{O}_X(-f)=0.$$
Therefore,  $$\h^1I_Z(f+K_X)=-\chi I_Z(C_0+(m-2b)f+K_X)=|Z|-\chi \mathcal{O}_X(f+K_X)=|Z|-1+g.$$ Hence $\ext^1(I_Z(f),\mathcal{O}_{X})=|Z|-1+g=c_2-1+g\neq 0$.

Any non-zero element in $ \Ext^1(I_Z(f),\mathcal{O}_{X})$ defines a non-trivial extension of type (\ref{extension3}).
Notice that by the Cayley-Bacharach property \cite[Theorem 12]{friedman},
$E$ is indeed a rank two vector bundle, with $c_1(E)=f$ and $c_2(E)=c_2>>0$.
Hence $\mathcal{G}$ is non-empty and by construction, for any $E\in \mathcal{G}$, $\h^0E\geq1$.

\vspace{0.3cm}

\vspace{0.4cm}
Now let us prove that any rank two vector bundle $E$ in $\mathcal{G}$ is stable with respect to $H\equiv C_0+\beta f$. 
To this end, since $E$ is a rank two vector bundle on $X$, we have to check that for any line subbundle $\mathcal{O}_X(G)\hookrightarrow E$ we have $\mu_H(\mathcal{O}_X(G))=G\cdot H < \mu_H(E)=\frac{c_1(E)\cdot H}{2}.$

By construction, $E$ is given by a non-trivial extension of type (\ref{extension3}). Therefore we have two possibilities:  
$\mathcal{O}_X(G)\hookrightarrow\mathcal{O}_X $ or  $\mathcal{O}_X(G)\hookrightarrow I_Z(f)$.

\vspace{0.4cm}
{\bf Case 1}
Assume that $\mathcal{O}_X(G)\hookrightarrow\mathcal{O}_X$. Then $$\mu_H(\mathcal{O}_X(G))\leq 0<\frac{1}{2}=\frac{c_1(E)\cdot H}{2}=\mu_H(E).$$

\vspace{0.4cm}
{\bf Case 2}
Assume that $\mathcal{O}_X(G)\hookrightarrow I_Z(f)$, where $G=sC_0+\ft f$ with $t=\deg(\ft)$.
In this case $$f-G=f-sC_0-\ft f$$ is an effective divisor, which by Lemma \ref{lema} implies that $s\leq0$ and $t\leq 1$. 

Notice that since $\h^0I_Z=0$, $f\neq G$.
Since $f-G$ is effective, we have  $(f-G)\cdot H>0$, which is equivalent to $t \leq s( e-\beta)$,  
and hence $$\mu_H(\mathcal{O}_X(G))=s(\beta- e)+t\leq 0 <\mu_H(E).$$

Thus, putting altogether, we have already seen that $E$ is $H$-stable.

\vspace{0.4cm}
We have proved that any $E$ in $\mathcal{G}$ is a rank two $H$-stable vector bundle with $\h^0E\geq1$, what implies that $\mathcal{G}\hookrightarrow W_H^1(2;f,c_2)$.
Thus $W_H^1(2;f,c_2)\neq \emptyset$.
 \end{proof}

Let us see now what happens for values of $k\geq2$.  
It follows from \cite[Proposition 7.7]{coskun} that $W_H^2(2;f,c_2)\neq \emptyset$.
We will prove the case $k\geq3$. To this end, we will use the following technical Lemma.
\begin{lemma}
Let $X$ be a ruled surface over a nonsingular curve $C$ of genus $g\geq0$ and $E$ a rank $2$ vector bundle with $c_1(E)=f$. If $E$ is stable respect to the ample divisor $H\equiv C_0+\beta f$, then $\h^0E\leq \h^0E|_{f}$.
\label{lemaf}
\end{lemma}
\begin{proof}
    First of all, let us see that $\h^0E(-f)=0$. If $\h^0E(-f)\neq0$, then  $\mathcal{O}_X(f)\hookrightarrow E$. Since $E$ is $H$-stable, we get $$\mu_H(\mathcal{O}_X(f))=1<\frac{1}{2}=\mu_H(E),$$ which is a contradiction. Hence $\h^0E(-f)=0$.

    If we twist the short exact sequence $$0\rightarrow\mathcal{O}_X(-f)\rightarrow\mathcal{O}_X\rightarrow \mathcal{O}_f\rightarrow0$$ by $E$ and we take cohomology, using the fact that $\h^0E(-f)=0$, we get $\h^0E\leq \h^0E|_{f}$.
\end{proof}

Now we are ready to study the Brill-Noether locus $W_H^k(2;f,c_2)$ for values of $k\geq3$.

\begin{theorem}

\label{thc1fv}
   Let $X$ be a ruled surface over a nonsingular curve $C$ of genus $g\geq0$, $c_2>>0$ an integer and $H\equiv C_0+\beta f$  an ample divisor on $X$. Then, $W_H^k(2;f,c_2)=\emptyset$ for all $k\geq 3$.
   \label{theodaniele}
\end{theorem}

\begin{proof}
    Let us assume that $W_H^k(2;f,c_2)\neq\emptyset$ and consider a vector bundle $E\in W_H^k(2;f,c_2)$. Since $\h^0E\geq k> 0$, we can take a non-zero  section $s$ of $E$. We denote by $Y$ its scheme of zeros and by $D= aC_0+\fb f$ the maximal effective divisor contained in $Y$. Then $s$ can be regarded as a section of $E(-D)$ and its scheme of zeros has codimension greater or equal than two. Thus, we have a short exact sequence
    \begin{equation}
        0\rightarrow \mathcal{O}_X(D)\rightarrow E\rightarrow I_Z(f-D)\rightarrow0
\label{extension4}
    \end{equation}
where $Z$ is a locally complete intersection $0$-cycle of length $|Z|=c_2-D(f-D)$.

\vspace{0.2cm}
Since $D$ is effective, by Lemma \ref{lema}, $a\geq 0$ and $b:=\deg(\fb)\geq0$.

Let us see that $D=0$.
Assume that $D\neq0$.
Since $D$ is effective, $D\cdot L>0$ for any ample divisor $L$ on $X$. 
On the other hand, since $E$ is stable respect to  $H\equiv C_0+\beta f$, 
we have  that $\mu_H(\mathcal{O}_X(D))<\mu_H(E).$

Putting altogether, we get $$0<\mu_H(\mathcal{O}_X(D))=D\cdot H<\frac{1}{2}=\mu_H(E),$$ which is a contradiction.
Hence $D=0$.

\vspace{0.2cm}
Therefore, by (\ref{extension4}), any vector bundle $E\in W_H^k(2;f,c_2)$ sits in a short exact sequence 
\begin{equation}
0\rightarrow \mathcal{O}_X\rightarrow E\rightarrow I_Z(f)\rightarrow0.
\label{extension5}
\end{equation}

If $\h^0I_Z(f)=0$, then $\h^0E=\h^0\mathcal{O}_X=1$.
On the other hand, if $\h^0I_Z(f)\neq0$ this would mean that the points of $Z$ are contained in a line.

Let $l$ be a line  such that $Z\cap l=\emptyset$. Restricting the exact sequence (\ref{extension5}) to $l$ we get the short exact sequence 
$$0\rightarrow\mathcal{O}_l\rightarrow E|_{l}\rightarrow \mathcal{O}_l\rightarrow0,$$
and taking cohomology we obtain $\h^0E|_{l}\leq 2$.

By Lemma \ref{lemaf}, this implies that $\h^0E\leq2$ and thus  $W_H^k(2;f,c_2)=\emptyset$ for all $k\geq 3$.
\end{proof}

Putting altogether:

\begin{corollary}
     Let $X$ be a ruled surface over a nonsingular curve $C$ of genus $g\geq0$, $c_2>>0$ an integer and $H\equiv C_0+\beta f$  an ample divisor on $X$. Then, $W_H^k(2;f,c_2)\neq\emptyset$ if and only if $1\leq k\leq 2$.
\end{corollary}
\begin{proof}
    It follows from Theorem \ref{thc1f}, \cite[Proposition 7.7]{coskun} and Theorem \ref{theodaniele}.
\end{proof}
    
See  \cite[Proposition 7.7]{coskun} for an alternative proof of Theorems \ref{thc1f} and \ref{thc1fv}.

\vspace{0.3cm}
Now we will analize in more detail the locus $W_H^1(2;f, c_2)\backslash W_H^2(2;f,c_2)$.

\begin{proposition}

\label{smooth}
Let $X$ be a ruled surface over a nonsingular curve $C$ of genus $g\geq0$, $c_2>>0$ an integer and $H\equiv C_0+\beta f$  an ample divisor on $X$. Then $W_H^1(2;f,c_2)\backslash W_H^2(2;f,c_2) $ is smooth and has the expected dimension, namely, 
$$ \rho_H^1=3c_2+g-1.$$
\end{proposition}

\begin{proof}

Let us take a vector bundle $E\in W_H^1(2;f,c_2)\backslash W_H^2(2;f,c_2)$. Since $\h^0E=1$, we can take a non-zero  section $s$ of $E$. We denote by $Y$ its scheme of zeros and by $D= aC_0+\mathfrak{b}f$ the maximal effective divisor contained in $Y$. Then $s$ can be regarded as a section of $E(-D)$ and its scheme of zeros has codimension greater or equal than two. Thus, we have a short exact sequence
\begin{equation}
\label{extension6}
  0\rightarrow \mathcal{O}_X(D)\rightarrow E\rightarrow I_Z(f-D)\rightarrow0  
\end{equation}
where $Z$ is a locally complete intersection $0$-cycle of length $|Z|=c_2-D(f-D)$.

\vspace{0.2cm}
Following the same arguments as in proof of Theorem \ref{theodaniele} we see that $D=0$.
Hence,  any vector bundle $E$ in $ W_H^1(2;f,c_2)\backslash W_H^2(2;f,c_2)$ sits in a short exact sequence of the following type
\begin{equation}
0\rightarrow \mathcal{O}_X\rightarrow E\rightarrow I_Z(f)\rightarrow0.
\label{sucesion1}
\end{equation}

Let us now check that any vector bundle $E$  in $ W_H^1(2;f,c_2)\backslash W_H^2(2;f,c_2)$ is a smooth point of $M_{X,H}(2;f,c_2)$. To this end, it is enough to see that $\h^2(E^{*}\otimes E)=0$.

\vspace{0.2cm}
First of all let us see that $\h^2(E^{*})=0$. If we tensor   the short exact sequence (\ref{sucesion1}) by $\mathcal{O}_X(K_X)$ and we take cohomology,
we get the long exact sequence $$0\rightarrow \Ho^0\mathcal{O}_X(K_X)\rightarrow \Ho^0E(K_X)\rightarrow \Ho^0I_Z(f+K_X)\rightarrow\cdots$$
From the fact that $\Ho^0\mathcal{O}_X(K_X)= \Ho^0I_Z(f+K_X)=0$, we deduce that $$0=\h^0E(K_X)=\h^2E^{*},$$ where the last equality follows from Serre duality.
Now we will see that $$\h^2(E^{*}\otimes \mathcal{O}_X(f)\otimes I_Z)=0.$$ Tensoring  the short exact sequence (\ref{sucesion1}) by $\mathcal{O}_X(K_X-f)$ and taking cohomology, we get the long exact sequence $$0\rightarrow \Ho^0\mathcal{O}_X(K_X-f)\rightarrow \Ho^0E(K_X-f)\rightarrow \Ho^0I_Z(K_X)\rightarrow\cdots$$

Since $\Ho^0\mathcal{O}_X(K_X-f)= \Ho^0I_Z(K_X)=0$,  we get $\Ho^0E(K_X-f)=0$ and by Serre's duality $\Ho^2E^{*}(f)=0$.
Hence $\h^2(E^{*}\otimes \mathcal{O}_X(f)\otimes I_Z)=\h^2(E^{*}\otimes \mathcal{O}_X(f))=0$.

 Finally, if we tensor  the short exact sequence  (\ref{sucesion1}) by $E^{*}$ and we take cohomology, we get $$\cdots \rightarrow \Ho^2(E^{*})\rightarrow \Ho^2(E^{*}\otimes E)\rightarrow \Ho^2(E^{*}\otimes \mathcal{O}_X(f)\otimes I_Z)\rightarrow0.$$ 
 and since $\Ho^2(E^{*})=\Ho^2(E^{*}\otimes \mathcal{O}_X(f)\otimes I_Z)=0$ we deduce that  $\Ho^2(E^{*}\otimes E)=0$. Hence $E$ is a smooth point of $M_H(2;f,c_2)$.

\vspace{0.1cm}
To end the proof, let us see that the map $$\mu_E:\Ho^0(E)\otimes \Ho^1(E^{*}\otimes\mathcal{O}_X(K_X))\rightarrow \Ho^1(E\otimes E^{*}\otimes \mathcal{O}_X(K_X))$$ is injective.
Since $E\in W_H^1(2;f,c_2)\backslash W_H^2(2;f,c_2)$,  $\Ho^0(E)\cong K$ and hence this is equivalent to prove that the map $$\Tilde{\mu}_E: \Ho^1(E^{*}\otimes\mathcal{O}_X(K_X))\rightarrow \Ho^1(E\otimes E^{*}\otimes \mathcal{O}_X(K_X))$$ is injective.

If we tensor  the exact sequence (\ref{sucesion1}) by $\mathcal{O}_X(-f)$  and we take cohomology, we get the long exact sequence $$\cdots \rightarrow \Ho^2\mathcal{O}_X(-f)\rightarrow \Ho^2E(-f)\rightarrow \Ho^2I_Z\rightarrow 0.$$
Since by duality $\h^2\mathcal{O}_X(-f)=\h^0\mathcal{O}_X(f+K_X)=0$  and $\h^2I_Z=\h^2\mathcal{O}_X=0$, we deduce that $\h^2E(-f)=0$ and this implies, again by Serre's duality, that $\h^0E^{*}(K_X+f)=0$. Moreover, if $\h^0E^{*}(K_X+f)=0$ we get  $\h^0E^{*}\otimes I_Z(K_X+f)=0$.

Therefore, since $\h^0(E^{*}\otimes I_Z(K_X+f))=0$, if we tensor the exact sequence (\ref{sucesion1}) by $E^{*}\otimes \mathcal{O}_X(K_X)$ and we take cohomology, we get $$0\rightarrow  \Ho^1(E^{*}\otimes\mathcal{O}_X(K_X))\xrightarrow{\Tilde{\mu}_E} \Ho^1(E\otimes E^{*}\otimes \mathcal{O}_X(K_X))\rightarrow\cdots$$ which implies that $\Tilde{\mu}_E$ is injective. Hence $\mu_E$ is injective.

\vspace{0.2cm}
We have proved that, for any vector bundle $E\in W_H^1(2;f,c_2)\backslash W_H^2(2;f,c_2)$, $E$ is a smooth point of $M_{X,H}(2;f,c_2)$ and $\mu_E$ is injective.
 
Therefore, by \cite[Corollary 2.6]{filimon},  $ W_H^1(2;f,c_2)\backslash W_H^2(2;f,c_2)$ is smooth and of the expected dimension at $E$, for all $E\in W_H^1(2;f,c_2)\backslash W_H^2(2;f,c_2)$. Hence  $ W_H^1(2;f,c_2)\backslash W_H^2(2;f,c_2)$ is smooth and of the expected dimension.
\end{proof}

\begin{corollary}
 All the irreducible components of $W_H^1(2;f,c_2)\backslash W_H^2(2;f,c_2) $ are smooth and of the expected dimension, namely, $\rho_H^1(2;f,c_2)=3c_2+g-1$.
\end{corollary}

\begin{remark}
The family of type (\ref{sucesion1}) used in the proof of Proposition \ref{smooth} parameterizes an irreducible component of $W_H^1(2;f,c_2)$ of the expected dimension $\rho_H^1(2;f,c_2)=3c_2+g-1$.
\end{remark}

\vspace{0.2cm}

\end{document}